\documentclass[3p]{elsarticle}

\usepackage{latexsym, amsmath, amssymb,amsthm, epsfig, color}

\usepackage{amsfonts,epsfig, subfigure} 
\usepackage{color,bm}

\begin{document}

\newcommand{\bdelta}{{\bm \delta}}
\newcommand{\DD}{ D}

\newcommand{\RR}{\mathbb R}
\newcommand{\ZZ}{\mathbb Z}
\newcommand{\CC}{\mathbb C}
\newcommand{\EE}{\mathbb E}
\newcommand{\NN}{\mathbb N}
\newcommand{\ZZP}{\mathbb{Z}_+}

\newcommand{\baa}{\begin{eqnarray*}}
\newcommand{\eaa}{\end{eqnarray*}}
\newcommand{\ba}{\begin{equation}}
\newcommand{\ea}{\end{equation}}

\newtheorem{Theorem}{Theorem}[section]
\newtheorem{theorem}[Theorem]{Theorem}
\newtheorem{lemma}[Theorem]{Lemma}
\newtheorem{proposition}[Theorem]{Proposition}
\newtheorem{corollary}[Theorem]{Corollary}
\newtheorem{definition}[Theorem]{Definition}
\newtheorem{example}[Theorem]{Example}
\newtheorem{remark}[Theorem]{Remark}

\newcommand{\eq}[1]{\begin{align*}#1\end{align*}}
\newcommand{\eqn}[1]{\begin{align}#1\end{align}}
\newcommand{\quo}[1]{\begin{quote}#1\end{quote}}
\newcommand{\bi}{\begin{itemize}}
\newcommand{\ei}{\end{itemize}}
\newcommand{\be}{\begin{enumerate}}
\newcommand{\ee}{\end{enumerate}}

\newcommand{\RED}{\textcolor{red}}
\newcommand{\BLUE}{\textcolor{blue}}

\newcommand{\x}{{\mathbf x}}
\newcommand{\y}{{\mathbf y}}
\newcommand{\bA}{{\mathbf A}}
\newcommand{\bAA}{\widetilde {\bA}}
 \newcommand{\bB}{{\mathbf B}}
\newcommand{\bD}{{\mathbf D}}
\newcommand{\bF}{{\mathbf F}}
\newcommand{\bK}{{\mathbf K}}
\newcommand{\bM}{{\mathbf M}}
\newcommand{\bP}{{\mathbf P}}
\newcommand{\bS}{{\mathbf S}}
\newcommand{\bW}{{\mathbf W}}
\newcommand{\bE}{{\mathbf E}}

\newcommand{\bFj}{{\mathbf F}^{[j]}}
\newcommand{\bg}{{\mathbf g}}
\newcommand{\bN}{{\mathbf N}}

\newcommand{\ou}{{\overline u}}
\newcommand{\of}{{\overline f}}

\newcommand{\bfa}{{\mathbf a}}
\newcommand{\bfb}{{\mathbf b}}
\newcommand{\bc}{{\mathbf c}}
\newcommand{\bd}{{\mathbf d}}
\newcommand{\pl}{p_\ell}
\newcommand{\bp}{{\bf p}}
\newcommand{\bq}{{\bf q}}
\newcommand{\ff}{{\mathbf f}}

\newcommand{\sak}{S_{\bA^{[k]}}}
\newcommand{\bPhi}{{\bm \Phi}}
\newcommand{\bI}{{\mathbf I}}
\newcommand{\Ht}{\widetilde{H}}
\newcommand{\bu}{{\mathbf u}}
\newcommand{\bv}{{\mathbf v}}

\newcommand{\xb}{\overline{x}}
\newcommand{\tb}{\overline{t}}

\newcommand{\uu}{{\mathbf u}}
\newcommand{\vv}{{\mathbf v}}
\newcommand{\rr}{{\mathbf r}}

\newcommand{\fhat}{{\hat{f}}}
\newcommand{\fshat}{{\hat{f''}}}
\newcommand{\ffh}{{\hat{\mathbf f}}}
\newcommand{\ffdh}{\widehat{\ff''}{ }}

\newcommand{\ggh}{{\hat{\mathbf g}}}
\newcommand{\gammah}{{\hat{\gamma}}}

\newcommand{\stil}{\widetilde{s}}
\newcommand{\sbar}{\overline{s}}

\newcommand{\gtil}{\widetilde{\gamma}}
\newcommand{\LL}{\mathcal{L}}
\newcommand{\Sbar}{\overline{S}}

\begin{frontmatter}

\title{A non-uniform corner-cutting subdivision scheme with
an improved accuracy }

\author{Byeongseon Jeong $^\S$, Hyoseon Yang ${}^\dagger$,
and Jungho Yoon ${}^\ddagger$}
\address{$^{\S}$ Institute of Mathematical Sciences, Ewha Womans University, Seoul 03760, South Korea
(bjeong\_ewha@ewha.ac.kr)}
\address{${}^\dagger$  Department of Computational Mathematics, Science and Engineering, Michigan State University, East Lansing, MI 48824, USA
(hyoseon@msu.edu, corresponding author)}
\address{${}^\ddagger$ Department of Mathematics, Ewha Womans University, Seoul 03760, South Korea
(yoon@ewha.ac.kr)}

\begin{abstract}
The aim of this paper is to construct a new non-uniform corner-cutting (NUCC) subdivision scheme that improves the accuracy of the classical (stationary and nonstationary) methods. The refinement rules are formulated via the reproducing property of exponential polynomials. An exponential polynomial has a shape parameter so that it may be adapted to the characteristic of the given data. In this study, we propose a method of selecting the shape parameter, so that it enables the associated scheme to achieve an improved approximation order (that is, {\em three}), while the classical methods attain the second order accuracy. An analysis of convergence and smoothness of the proposed scheme is conducted. The proposed scheme is shown to have the same smoothness as the classical Chaikin's corner-cutting algorithm, that is, $C^1$. Finally, some numerical examples are presented to demonstrate the advantages of the new corner-cutting algorithm.
\end{abstract}


\end{frontmatter}


\section{Introduction}

%



Subdivision schemes are recursive processes of generating smooth curves
or surfaces from a given sequence of control points.
By applying a set of refinement rules to an initial
sequence of points iteratively, a subdivision scheme computes denser
sequences of points under a suitable convergence criteria.
Due to the algorithmic efficiency in geometric modeling and design,
subdivision schemes are extensively studied and utilized in many research fields
such as computer-aided geometric design, computer graphics, image processing, scientific computing etc.
General discussions on subdivision schemes may be found in the references
\cite{CDM-91, D-92, DL-02}.

The smoothness and approximation order of limit functions obtained
by the recursive refinement process are two key features
of a convergent subdivision scheme.
They are strongly related to the generation or reproduction properties
of a class of functions such as
a set of algebraic polynomials
\cite{CDM-91, CC-78, CC-13, CLR-80, CH-11, DD-89, DHSS-08, L-03}
and exponential polynomials \cite{DL-91, DLL-03, R-88}.
Mostly well-known subdivision schemes based on algebraic
polynomials are the B-spline \cite{CC-78, CLR-80}
and Delauriers-Dubuc interpolatory schemes \cite{DD-89},
which respectively provide optimal smoothness and approximation order.
Recent studies have reported that exponential polynomials
are useful and effective in designing
and approximating various geometric shapes.
The refinement rules of the associated subdivision scheme are level-dependent,
in which case we call the scheme {\em nonstationary}.
Since the earlier works as \cite{DL-91, DLL-03, R-88} on
exponential B-splines,
important theoretical progress has been achieved during the last two decades,
especially in regards to
the reproducing and generation properties of exponential polynomials
\cite{CCR-14, FMW-10, JKLY-13} and the construction of  wavelets
\cite{KU, UB,VTU}.
Also, important practical subdivision
algorithms have been reported. The readers are referred to
\cite{BCR-07, CGR-11, CGR-16, CR-11, DLL-03, JSD-03, JLY-13}) and the references therein.

The quadratic B-spline,
known as Chaikin's corner-cutting algorithm \cite{C-74},
is one of the most popular
and fundamental scheme with various extensions and applications
\cite{B-87, DY-2019, DS-78, GQ-96, R-15, TP-20}.
The exponential B-spline scheme of degree $2$ is its
natural nonstationary extension. It has advantages because it can reproduce
geometrically important shapes in CAGD.
Such merits are attributed to the ability to reproduce
exponential polynomials.
At this point, one should notice that
an exponential polynomial has a shape parameter that can be
tuned to the characteristic of the given data.
Thus, the final approximation may be more faithful to the data
initially given.
Motivated by this observation,
this study aims to introduce a new non-uniform corner-cutting
algorithm that generalizes the exponential B-spline of degree 2.
In particular, we propose an optimal method of selecting
the shape parameter in the exponential polynomial, so that
the proposed scheme achieves the improved approximation
order (that is, {\em three}),
while the classical corner-cutting (stationary and nonstationary) methods
have the second order accuracy.
(e.g., see \cite{JKLY-13, L-03}).
Further, we show that the proposed scheme provides
the same smoothness (that is, $C^1$)
as the classical Chaikin's corner-cutting algorithm.
Finally, some numerical examples are presented to demonstrate the
advantages of the new scheme.

The remaining contents of this study are organized as follows.
Section \ref{PRE} is devoted to provide some basic  concepts of (non-uniform)
subdivision schemes and related basic notions.
In Section \ref{SEC-CONST}, the new subdivision masks are constructed via the
local reproducing property of exponential polynomials, along with
a method for selecting the shape parameter in the exponential polynomial.
The convergence and smoothness of the proposed scheme are investigated
in Section \ref{SEC-CS}.
Also, in Section \ref{SEC-AO}, we analyze its approximation order.
Finally, some numerical examples supporting the theoretical results
are presented in Section \ref{SEC-NE}.

\section{Preliminaries} \label{PRE}

In this section we present some necessary notation
and background material.
Let $\ZZ$ be the set of all integers, and denote by $\ZZP$ the set
of nonnegative integers.
For a given sequence $\ff^k := \{ \ff^k_j : j \in \ZZ \}$,
the refinement rule of a non-uniform subdivision scheme
$\{ S_k : k \in \ZZP \}$ is defined by
\eqn{ \label{RR}
\ff^{k+1}_{2j+\nu} := S_k \ff^k_{2j+\nu}
:= \sum_{i \in \ZZ} \bfa^{j,k}_{2i+\nu} \ff^k_{j-i},\quad
\nu = 0,1.
}
This can be written formally as $\ff^{k+1} := S_k \ff^k$.
The sequence $\bfa^{j,k} := \{ \bfa^{j,k}_{n} : n \in \ZZ \}$ is called
the subdivision {\em mask} and is assumed to be finitely supported.
In fact, some fundamental
characteristics of a subdivision are embedded in its symbol
associated with the mask $\bfa^{j,k}$:
\eq{a^{j,k}(z) := \sum_{n \in \ZZ} \bfa^{j,k}_n z^n.
}
Since the mask $\bfa^{j,k}$ is finitely supported, the symbol
$a^{j,k}(z)$ becomes a Laurent polynomial.
The norm of a subdivision operator $S_k$ associated with the refinement
masks $\{ \bfa^{j,k} : j \in \ZZ \}$ is defined by
\eq{
\| S_k \|_\infty = \sup_{j \in \ZZ} \Big\{ \sum_{i \in \ZZ} |\bfa^{j,k}_{2i+\nu} |:  \nu = 0,1 \Big\}.
}
It is of necessity to recall the concept of parametrization for
a subdivision scheme.
For each level $k$, the value $\ff^k_j$  can be attached to a grid point
value $t^k_j := 2^{-k}j$ or $2^{-k}(j-1/2)$, 
which are respectively called {\em primal} and {\em dual} parametrization.
The value $k$ indicates the scale parameter while $j-\frac 12$
(or $j$) stands
for the location parameter.
For a given parametrization $\{t^k_j\}$,
a subdivision scheme $\{ S_k \}$ is {\em convergent} if for any initial data $\ff^0$,
there exists a function $f \in C(\RR)$ such that for any compact subset $\Omega$ of $\RR$,
\eq{
\lim_{k \to \infty} \sup_{j \in \ZZ \cap 2^k \Omega} | \ff^k_j - f(t^k_j) | =0
}
and $f \not\equiv 0$ for some $\ff^0$.
Also, a subdivision scheme $\{ S_k \}$ is said to be {\em stable} if
there exists a constant $C > 0$ such that
for any $\ff \in \ell^\infty(\ZZ)$,
\eq{
\sup_{k,n\in\ZZP}\| S_{k+n}\cdots S_{k}\ff \|_\infty \leq C \|\ff\|_\infty.
}

In this study, the term `reproduction' is used to indicate the capability
of a subdivision scheme of reproducing the functions from which
the initial sequence of data values are sampled.
Below, we specify the definition of exponential polynomial reproduction.

\begin{definition}
{\rm
Let $\EE$ be a space of exponential polynomials.
We say that a subdivision scheme $\{ S_k \}$ {\em reproduces} the
functions in $\EE$ if
for the initial data $\ff^0 := \{ f(t^0_j) : j \in \ZZ \}$
with $f\in\EE$, each sequence $\ff^k := S_{k-1}\cdots S_0 \ff^0$
with $k \geq 1$
satisfies the relation $\ff^k_j = f(t^k_j)$ for any $j \in \ZZ$.
}
\end{definition}

\def\vp{\varphi}
\noindent
It is well-known that the exponential B-splines can reproduce
at most two exponential polynomials if a suitable normalization
factor is taken. Hence, this study is focused on the
following type of space
$$ \EE_2:={\rm span} \{ \vp_0, \vp_1\}, $$
where $\varphi_0$ and $\varphi_1$ are linearly independent
exponential polynomials.
A basic requirement for $\{\vp_0, \vp_1\}$ is that its Wronskian
matrix at $\vartheta =0$ is invertible, that is,
$$ 
{\rm det} \big (\vp^{(\beta)}_\alpha(0):\alpha,\beta=0,1 \big)
\not= 0.
$$
It is clear that if this condition is satisfied,
then the functions $\vp_0$ and $\vp_1$ are linearly independent.

In addition, for later use, we introduce some notations.
For a given  sequence $\ff^k := \{ \ff^k_j : j \in \ZZ \}$
with density $2^{-k}$,
$\Delta \ff^k$ indicates its second-order difference defined by
$$\Delta \ff^k_j := 2^{2k}(\ff^k_{j-1} - 2 \ff^k_j + \ff^k_{j+1}).$$
Also, denote $\nabla \ff^k_{j+1} = 2^k (\ff^k_{j+1} -\ff^k_j)$.
In addition, letting $\ff^0$ be the initial sequence of data,
put $\bd^0 =\Delta \ff^0$, and then, for each refinement level $k\in\NN$,
$\bd^k$ is defined  as
\eqn{ \label{DK}
\bd^k := S^k_a \bd^0 = S^k_a\Delta  \ff^0,
}
where $S_a$ is the subdivision of the B-spline of degree $2$,
i.e., the classical corner-cutting method.

\section{Nonuniform corner-cutting subdivision scheme} \label{SEC-CONST}

The objective of this section is to construct a novel non-uniform corner-cutting
(referred to as `NUCC') subdivision scheme.
The new subdivision masks are defined via the property of
reproducing two exponential polynomials $\{\vp_0, \vp_1\}$
based on the dual parametrization.
The exponential polynomials $\{\vp_0, \vp_1\}$
and the internal shape parameter are adjusted
by reflecting the local data feature to improve the accuracy.
To this end, for a given data $\ff^k_i$ at level $k$, we assume
without great loss that $\ff^k_i$ and $\nabla \ff^k_{i+1}$
do not vanish simultaneously. If this is not the case,
one may put the internal shape parameter as zero, then the scheme becomes
the classical corner-cutting method.

\medskip \noindent
$\bullet$\ {\bf Non-uniform Subdivision Mask.}
For a given sequence of data $\ff^k$ at level $k$,
we first assume that $\ff^k_j$
are nonzero (practically, away from zero).
The subdivision mask $\bfa^{j,k} := \{ \bfa^{j,k}_{n} : n \in \ZZ \}$
is formulated via the reproducing property of two
exponential polynomials $\{\vp_0, \vp_1\}$
at the evaluation point, say
$\tb := t^{k+1}_\ell$ (where $\ell =2j , 2j+1$).
According to the parity of location $\ell$,
we determine the two sets of nonzero coefficients: even mask
$\{ \bfa^{j,k}_{-2},\ \bfa^{j,k}_{0} \}$
and odd mask $\{ \bfa^{j,k}_{-1},\ \bfa^{j,k}_{1} \}$.
To derive the even mask, consider the following local reproduction property
of the two exponential functions $\{\vp_0,\vp_1\}$:
\eqn{ \label{LS-1}
\LL_0(\tb) \varphi_n(t^k_j) + \LL_1(\tb) \varphi_n(t^k_{j+1})
= \varphi_n(\tb),\quad n = 0,1.
}
where $\mathcal{L}_0$ and $\mathcal{L}_1$  are the Lagrange functions
in the space $\EE_2 ={\rm span}\{\vp_0, \vp_1\}$.
Representing this linear system  explicitly yields the matrix form
\eq{
\begin{bmatrix}
\varphi_0(t^k_j) & \varphi_0(t^k_{j+1}) \\
\varphi_1(t^k_j) & \varphi_1(t^k_{j+1}) \\
\end{bmatrix}
\begin{bmatrix}
\LL_0(\tb)  \\
\LL_1(\tb)
\end{bmatrix}
=
\begin{bmatrix}
\varphi_0(\tb) \\
\varphi_1(\tb)
\end{bmatrix}.
}
The even mask of the proposed subdivision scheme
(associated to the location $\tb$)  is defined as
the solution of this linear system.
To be  more precise,
letting $\tb := t^{k+1}_{2j}$ be an evaluation point at level $k+1$,
we select
$$\{\vp_0, \vp_1\} :=
\{\exp\big(\gamma_{2j} (\cdot - \tb) \big),
\exp\big(-\gamma_{2j} (\cdot - \tb) \big)\},$$
where $\gamma_{2j}$ is real or pure imaginary.
Then the solution of the above linear system
can be formulated specifically as
\eqn{ \label{L-0-1}
\bfa^{j,k}_{0}:= \LL_0(\tb) = \frac{\sinh(\frac 34 \gamma_{2j}2^{-k})}
                {\sinh(\gamma_{2j} 2^{-k})},\quad
\bfa^{j,k}_{-2}:= \LL_1(\tb) = \frac{\sinh(\frac 14 \gamma_{2j}2^{-k})}
                {\sinh(\gamma_{2j} 2^{-k})}.
}
The odd mask is defined in a similar way.
Letting $\tb = t^{k+1}_{2j+1}$,
we consider the reproduction of two exponential functions
$\{\vp_0, \vp_1 \}=
\{\exp\big(\gamma_{2j+1} (\cdot - \tb) \big),
\exp\big(-\gamma_{2j+1} (\cdot - \tb) \big)\}$.
Under this setting,
the odd mask is also formulated as the solution of
the linear  system \eqref{LS-1}:
\eqn{ \label{L-0-2}
\bfa^{j,k}_{1}
= \frac{\sinh(\frac 14 \gamma_{2j+1}2^{-k})}
                {\sinh(\gamma_{2j+1} 2^{-k})},\quad
\bfa^{j,k}_{-1}
= \frac{\sinh(\frac 34 \gamma_{2j+1}2^{-k})}
                {\sinh(\gamma_{2j+1} 2^{-k})}.
}
As a conclusion, the new subdivision mask is of the form
\eqn{ \label{MASK}
\bfa^{j,k} =\Big \{
\frac{\sinh(\frac 14 \gamma_{2j}2^{-k})}
                {\sinh(\gamma_{2j} 2^{-k})}, \
\frac{\sinh(\frac 34 \gamma_{2j+1}2^{-k})}
                {\sinh(\gamma_{2j+1} 2^{-k})}, \
\frac{\sinh(\frac 34 \gamma_{2j}2^{-k})}
                {\sinh(\gamma_{2j} 2^{-k})}, \
\frac{\sinh(\frac 14 \gamma_{2j+1}2^{-k})}
                {\sinh(\gamma_{2j+1} 2^{-k})}
\Big \}.
}
In this study, we especially suggest to choose the shape parameter
$\gamma_{2j+\nu}$ for $\nu \in \{ 0,1 \}$ as
\eqn{ \label{GAM-EVEN}
\gamma_{2j+\nu} := \gamma_{2j+\nu, k} :=
\sqrt{\frac{\bd^k_{j+\nu}} {\ff^k_{j+\nu} + \epsilon}},
\qquad {\rm sign}(\ff^k_{j}) = {\rm sign}(\epsilon)
}
where a nonzero number $\epsilon$ is employed
to prevent the denominator becoming too small or zero.

\begin{remark}
{\rm
Suppose that the initial data values are sampled from a smooth function
$f$. Let $\tb = t^{k+1}_{2j+\nu}$ be an evaluation point at level $k+1$.
It can be easily deduced from \eqref{GAM-EVEN} that
$$ \gamma^2_{2j+\nu} \approx  {\frac {f''(\tb)} {f(\tb)}}.$$
We will see later that this choice of the shape parameter enables
the proposed scheme to provide improved order of accuracy.
For better readability of this paper, the specific discussion
on the  motivation for this  choice $\gamma_{2j+\nu}$ is postponed
to Section \ref{SEC-AO}.
}
\end{remark}

\begin{remark} \label{RMK-EP}
{\rm 
We suggest to choose the non-zero value $\epsilon$ depending on 
the density of the given initial data. 
For instance, one may put  $\epsilon$ as $|\epsilon |=2^{-2k_0}$. 
Accordingly, the shape parameter  $\gamma_{j}$ can 
be  uniformly bounded independent of $j$.
}
\end{remark}

This study is mainly focused on the mask in \eqref{MASK}. However,
the denominator in \eqref{GAM-EVEN} is zero or very small,
the refinement rule may be amended as below.

\medskip \noindent
$\bullet$ {\bf  Alternative Approach.}
If $\ff^k_{j+\nu}$ ($\nu=0,1$) is zero or very close to zero,
an alternative way to construct
the subdivision masks is given as follows.
Let $\tb=t^{k+1}_{2j+\nu}$ be the grid  point associated to
$\ff^{k+1}_{2j+\nu}$.
Then we consider the space of exponential polynomials
$$\EE_2 ={\rm span} \{1, \exp (\gamma_{2j+\nu}(\cdot -\tb) )\} $$
and then select the shape parameter as
\eqn{ \label{GAM-ODD}
\gamma_{2j+\nu} := \gamma_{2j+\nu,k} :=
{\frac{\bd^k_{j+\nu}}{\nabla \ff^k_{j+1}+\epsilon}},
\qquad \nabla \ff^k_{j+1} = 2^k (\ff^k_{j+1} - \ff^k_{j}).
}
Accordingly, by following the same technique as above,
the new subdivision mask can be derived as the form
\eqn{ \label{MASK-20}
\bfa^{j,k}_{-2} = \frac{e^{\gamma_{2j} \frac 14 2^{-k}}-1}{e^{\gamma_{2j} 2^{-k}} - 1},\quad
\bfa^{j,k}_{-1} = \frac{e^{\gamma_{2j+1} \frac 34 2^{-k}}-1}{e^{\gamma_{2j+1} 2^{-k}} - 1},\quad
\bfa^{j,k}_{0} = 1 - \bfa^{j,k}_{-2},\quad
\bfa^{j,k}_1 = 1 - \bfa^{j,k}_{-1}.
}
One should note that the above even and odd rules
satisfy the partition of unity, respectively. Hence,
for later use, it is necessary to remark that
its Laurent polynomial $a^{j,k}(z)$
satisfies the equations
\ba\label{LAU-S2}
a^{j,k} (1) = \sum_{n\in\ZZ} \bfa^{j,k}_n =2,
\qquad
a^{j,k} (-1) =
\sum_{n\in\ZZ} \bfa^{j,k}_n (-1)^n=0.
\ea

\begin{remark} \label{RMK-1}
{\rm
It is obvious  that the mask $\bfa^{j,k}$ converges to
${\bf a}=\{\frac 14, \frac 34, \frac 34 , \frac 14 \}$,
as $k \to \infty$,  which is the mask of B-spline of degree $2$
(known as Chaikin's corner cutting method).
If $\bd^k_{j} $ is zero, then
the NUCC scheme turns out to be  the corner-cutting method.
So, the NUCC method is nonlinear and it can be understood 
as a perturbation of the Chaikin's corner-cutting algorithm.
In this regard, without great loss,
we assume throughout this paper that
$$-\frac 14 + {\bf a}_n < \bfa^{j,k}_n < {\bf a}_n + \frac 14$$
for the mask $\bfa_n$ of the Chaikin's algorithm.
}
\end{remark}

\section{Analysis of convergence and smoothness} \label{SEC-CS}

In this section, we discuss the convergence and the smoothness of
the NUCC scheme $\{S_k\}$ with the masks defined in \eqref{MASK}.
Our specific goal is to prove that the scheme  $\{S_k\}$
generates $C^1$ limit curves.  Our analysis basically relies on
the concept of the `asymptotic equivalence' relation
between two subdivision schemes.

\medskip
\begin{definition}
{\rm
A non-uniform subdivision scheme $\{ S_k \}$ with the mask
$\{ \bfa^{j,k}_{n}\}$ is said to be
{\em asymptotically equivalent} to a uniform stationary scheme $S$ with
the mask $\{\bfa_n\}$  if
\eqn{ \label{AE}
\sum_{k \in \ZZP} \sup_{j \in \ZZ} \sum_{n \in \ZZ}
|\bfa^{j,k}_{n} - \bfa_n | < \infty.
}}
\end{definition}

\noindent
Based on this concept,
we first show that the NUCC scheme is stable and convergent.
To do this, we cite a  result \cite[Proposition 3.2]{DLY-14}.

\begin{lemma} \label{LEM-AE} {\rm \cite{DLY-14}}
Suppose that a non-uniform subdivision scheme $\{ S_k \}$ is asymptotically
equivalent to a uniform stationary scheme $S$.
If $S$ is stable and convergent,
then so is the scheme $\{ S_k \}$.
\end{lemma}

\begin{theorem} \label{THM-CONV}
The NUCC scheme $\{ S_k \}$
is stable and convergent, that is $C^0$.
\end{theorem}
\begin{proof}
Using the Maclaurin series expansion of
$\sinh (\gamma x)$
which appears in the explicit form of the  mask
$\bfa^{j,k}_{n}$ in \eqref{MASK},
it is clear that
\eqn{ \label{TE-1}
|\bfa^{j,k}_{n} - \bfa_n|\leq c \gamma_{2j+\nu}^2 2^{-2k},\quad \forall j \in \ZZ,
\quad \nu = 0,1
}
for the mask $\bfa_n$ of the quadratic B-spline scheme.
In view of the discussion  in Remark \ref{RMK-EP}, 
$\gamma_{2j+\nu}$ is uniformly bounded. Hence, 
the relation \eqref{AE} is readily satisfied.
Similarly, we can prove
$|\bfa^{j,k}_{n} - \bfa_n|\leq c\gamma_{2j+\nu} 2^{-k}$ for the mask in \eqref{MASK-20}.
It ensures that
the NUCC scheme  $\{ S_k \}$ is asymptotically
equivalent to
the quadratic B-spline scheme $S_\bfa$.
Since $S_\bfa$ is stable and convergent, by Lemma \ref{LEM-AE},
the scheme $\{ S_k \}$ is also stable and convergent.
\end{proof}

To analyze the smoothness of a  scheme $\{ S_k \}$,
we use the so-called `Property A' introduced in \cite{DLY-14},
which is described in terms of the Laurent polynomial
\eqn{ \label{DELTA}
\DD^{j,k}_m(z) :=\sum_{\ell =0}^m
(-1)^\ell {m\choose \ell}z^\ell a^{j-\ell,k}(z),
}
where $a^{j,k}(z)$ is the symbol associated to the mask
$\{\bfa^{j,k}_{n}:n\in\ZZ\}$.

\begin{definition}
{\rm
A non-uniform subdivision scheme $\{ S_k \}$ is said to satisfy
the {\em Property A of order m} if
\eq{
\sum_{k = 0}^\infty
2^{k(m-\alpha)}
\Big|\frac{d^\alpha \DD^{j,k}_m}{dz^\alpha}(\pm 1)\Big| < \infty,\quad 0 \leq \alpha < m.
}
}
\end{definition}

A sufficient condition for the smoothness of a non-uniform subdivision scheme
is discussed in \cite{DLY-14}.

\begin{theorem} {\rm \cite{DLY-14}} \label{THM-A}
Suppose that  a non-uniform binary scheme $\{S_k\}$
is asymptotically equivalent to a stationary
scheme $S_a$.
If $\{S_k\}$ satisfies Property A of orders $1\le\ell\le m$ and
$S_a$ is a $C^m$ scheme with a stable basic limit function, then
$\{S_k\}$ is also $C^m$.
\end{theorem}

\begin{theorem}
The NUCC scheme $\{ S_k \}$ is $C^1$.
\end{theorem}
\begin{proof}
We first prove that
the NUCC  satisfies Property A of order $1$. To this end,
we need to estimate the Laurent
polynomial
$\DD^{j,k}_1(z) = a^{j,k}(z) - z a^{j-1,k}(z)$ at $z=-1,1$.
Let  $a(z)$ be the symbol of the quadratic B-spline.
It follows that
\eq{
|\DD^{j,k}_1(\pm 1)|
\leq |a^{j,k}(\pm 1)-a(\pm 1)|
+
|a(\pm 1) - (\pm 1) a^{j-1,k}(\pm 1)|.
}
where  $a(1) = 2$ and $a(-1) = 0$.
Here, to estimate the above terms, we need to consider the Laurent
polynomial $a^{j,k}(z)$ associated to both of the cases \eqref{MASK}
and \eqref{MASK-20} respectively.
It is easy to see that
$a^{j,k}(z)$ of \eqref{MASK}  satisfies
$$a^{j,k}(\pm 1) = a(\pm 1) +O(2^{-2k}).$$
Also, as observed in \eqref{LAU-S2},
the mask of the NUCC  of \eqref{MASK-20}  satisfies
the partition of unity
such that  $a^{j,k}(1) = 2$ and $a^{j,k}(-1) = 0$.
So, clearly, $a^{j,k}(\pm 1) = a(\pm 1)$.
Consequently, it is immediate that
$$|\DD^{j,k}_1(\pm 1)| \leq c2^{-2k}$$ for a constant $c>0$.
It verifies that  the NUCC scheme fulfills the Property A.
Moreover, 
as shown in the proof of Theorem \ref{THM-CONV},
the NUCC scheme $\{ S_k \}$ is asymptotically equivalent to
the quadratic B-spline subdivision $S_a$.
Since $S_a$ is $C^1$ and its basic limit functions is stable,
by Theorem \ref{THM-A},
$\{ S_k \}$ is also $C^1$. It completes the proof.
\end{proof}

\begin{remark}
{\rm
A bivariate subdivision scheme for modeling surfaces can be easily
constructed via the tensor product of the univariate schemes.
Hence, it clearly follows from \cite[Corollary 1]{CD-18}
that the tensor product
of univariate NUCC schemes is $C^1$.
}
\end{remark}

\section{Approximation order} \label{SEC-AO}

This section aims to show that the NUCC scheme
$\{ S_k \}$ achieves the approximation order $3$,
while the classical methods provide the second order accuracy.
Suppose that the initial data is of the form
$\ff^{0}=\{f(2^{-k_0}(n-\frac 12)):n\in\ZZ\}$ with $k_0\in\ZZ_+$.
We shall show that
$$\|\ff^\infty- f \|_\infty \leq c_f 2^{-3k_0}$$
where
$\ff^\infty= \lim_{k\to\infty}S_k \cdots S_0 \ff^{0}$.
This work especially focuses on approximating functions
in the Sobolev space
$$W^r_\infty(\RR) := \big \{ f: \RR \to \RR : \| f \|_{r,\infty} := \sum_{n = 0}^{r} \| f^{(n)} \|_\infty < \infty \big \}$$
for $r \in \ZZP.$
Since the proposed scheme is data-dependent,
the techniques commonly used for the proof in the uniform case
are not applicable to our case. Eventually, the technique for this proof is
more involved.
In order to facilitate our proof, throughout this section,
we use the following notation.
For a given $k \in \ZZ_+$,
denote by
\eqn{ \label{FH}
\ffh^{k} := \big \{ f(2^{-k_0} t^{k}_j) : j \in \ZZ \big \}
}
the sequence sampled from  a smooth function $f(2^{-k_0}\cdot)$
at the $k$th level grid points $t^k_j$.
In the following lemma,
we discuss the approximation properties of the subdivision operator  $S_k$
to the sequence of data $\ffh^{k}$.

\begin{lemma} \label{LEM-EST-S-0}
For a given function $f \in W^3_\infty(\RR)$, let $\ffh^{k}$
be the sequence of data of the form  \eqref{FH}.
If $\{ S_k \}$ is the NUCC scheme
acting on the initial data of the form
$\ff^{0}=\{f(2^{-k_0}(n-\frac 12)):n\in\ZZ\}$, then we have
\ba\label{LEM-10}
| S_{k} \ffh^{k}_{j} - \ffh^{k+1}_{j} | \leq c_f 2^{-2k_0-k},
\quad  j\in \ZZ,
\ea
with a constant $c_f>0$ depending on $f$.
\end{lemma}
\begin{proof}
Letting $S_\bfa$ be the subdivision operator of the quadratic B-spline scheme
with the mask $\bfa$,
we can express
\eqn{ \label{EXPR-01}
\begin{aligned}
S_{k} \ffh^{k}_{j} - \ffh^{k+1}_{j}
=\big ( S_{k} \ffh^{k}_{j} - S_{\bfa} \ffh^{k}_{j}\big ) +
\big (S_{\bfa} \ffh^{k}_{j} - \ffh^{k+1}_{j}\big )
\end{aligned}
}
Since the quadratic B-spline $S_\bfa$
provides the second-order accuracy for smooth functions and
the data $\ffh^k$ has the density $2^{-k_0-k}$,
it holds clearly that
$|S_{\bfa} \ffh^{k}_{j} - \ffh^{k+1}_{j}| = O(2^{-2(k_0+k)})$.
Thus, it remains to estimate the term
$|S_{k} \ffh^{k}_{j} - S_{\bfa} \ffh^{k}_{j}|$.
From the proof of Theorem \ref{THM-CONV} and Remark \ref{RMK-EP},
we have
\begin{align*}
|S_{k} \ffh^{k}_{j} - S_{\bfa} \ffh^{k}_{j}|
 \leq \|\bfa^{j,k} - \bfa\|_\infty \|\ffh^k\|_\infty
\leq c 2^{-k} \| \bd^k\|_\infty \|\ffh^{k}\|_\infty
\end{align*}
with a constant $c > 0$.
By construction, $\bd^k = S^k_\bfa \Delta \ff^0$.
Since the initial data $\ff^0$ is sampled from the dilated 
function $f(2^{-k_0}\cdot)$,
it is clear that
$|\Delta \ff^0_n| \leq  |2^{-2 k_0} f''(2^{-k_0}(n-\frac 12))| 
+c'2^{-3k_0}\|f'''\|_\infty$
for some constant $c' > 0$.
It concludes
$|S_{k} \ffh^{k}_{j} - S_{\bfa} \ffh^{k}_{j}| \leq c_f 2^{-2k_0-k}$
with $c_f>0$ independent of $k$.
Hence, we obtain the lemma's claim.
\end{proof}

\begin{lemma} \label{LEM-EST-F}
Suppose that  the initial data
$\ff^{0} := \{f(2^{-k_0}(n - \frac 12)) : n \in \ZZ \}$ is sampled
from a function $f\in W^3_\infty(\RR)$, 
and let $\ffh^{k}$ be the sequence of data of the form  \eqref{FH}.
If $\ff^k = S_{k-1}\cdots S_{0} \ff^{0}$
with the NUCC scheme $\{S_k\}$, we obtain
\eq{
\| \ff^k - \ffh^{k} \|_\infty \leq c_f 2^{-2k_0},
}
where $c_f>0$ is a positive constant depending on $f$ but independent of
$k_0$.
\end{lemma}
\begin{proof}
We can represent $\ff^k$ as the following telescoping sum:
\eq{
\ff^k = S_{k-1} \ffh^{k-1} + \sum_{\ell = 1}^{k-1} S_{k-1}\cdots S_{k-\ell}
(S_{k-\ell-1} \ffh^{k-\ell-1} - \ffh^{k-\ell}).
}
According to Theorem \ref{THM-CONV}, the NUCC scheme $\{ S_k \}$ is stable
such that
$\| S_{k-1}\cdots S_{k-\ell} \|_\infty \leq c$ for any $k$ and $\ell$
with $k \geq \ell$.
It induces the relation
\eq{
| \ff^k_j - \ffh^{k}_j |
&\leq
\| S_{k-1} \ffh^{k-1} - \ffh^k \|_\infty +
\sum_{\ell = 1}^{k-1}
\| S_{k-1}\cdots S_{k-\ell} \|_\infty
\| S_{k-\ell-1} \ffh^{k-\ell-1} - \ffh^{k-\ell} \|_\infty  \\
&\leq
c  \sum_{\ell = 0}^{k-1} \| S_{\ell} \ffh^{\ell} - \ffh^{\ell+1} \|_\infty.
}
Further, by Lemma \ref{LEM-EST-S-0},
$\| S_{\ell} \ffh^{\ell} - \ffh^{\ell+1} \|_\infty \leq
c_f 2^{-2k_0-\ell}$. Hence, it leads to  the bound
\eq{
\sum_{\ell = 0}^{k-1} \| S_{\ell} \ffh^{\ell} - \ffh^{\ell+1} \|_\infty
\leq c_f \sum_{\ell = 0}^{k-1} 2^{-2 k_0-\ell}
\leq c_f \sum_{\ell = 0}^{\infty} 2^{-2 k_0-\ell}
\leq c'_f 2^{-2k_0},
}
which finishes the proof.
\end{proof}

Based on the above results, we shall investigate the approximation order
of the NUCC scheme.  For better readability of this paper,
our proof will be done by focusing on the mask in \eqref{MASK};
the other case in \eqref{MASK-20} can be done similarly.
In advance to proceed further,  we introduce some notation.
For each $j\in\ZZ$ and $\nu=0,1$,
we use the abbreviation
$$ j_\nu := 2j+ \nu.$$
Denote by
\eqn{ \label{FH-D2}
\ffdh^{k} := \big \{ f''(2^{-k_0} t^{k}_j) : j \in \ZZ \big \}
}
the sequence sampled from  a smooth function $f''(2^{-k_0}\cdot)$
at the $k$th level grid points $t^k_j$.
Let $\tb =2^{-k_0} t^{k+1}_{j_\nu}$ with $\nu=0,1$
be the evaluation points
at level $k+1$ between $2^{-k_0}t^k_j$ and $2^{-k_0}t^k_{j+1}$.
Then, for a given smooth function  $f$,
consider an auxiliary function $Q_\nu f$
defined by
\eqn{ \label{QF}
Q_\nu f (t) = \LL_{0,\nu}(t) f(2^{-k_0}t^k_j) + \LL_{1,\nu}(t) f(2^{-k_0}t^k_{j+1})
}
where $\LL_{0,\nu}$ and $\LL_{1,\nu}$ are the
Lagrange functions as given in \eqref{LS-1}
in the exponential space
\ba\label{E-SP}
\EE_2 := {\rm span} \{ \exp(\gammah_{j_\nu} (\cdot - \tb)),  \
\exp(-\gammah_{j_\nu} (\cdot - \tb)) \}
\ea
satisfying
$\LL_{n,\nu} (t^k_{j+\ell})= \delta_{\ell, n}$ for $\ell, n=0, 1$.
Here, we suppose that the shape parameter is chosen as
\eqn{ \label{GAM-H}
\gammah_{j_\nu} := \gammah_{j_\nu,k} :=
\sqrt{ \frac{\ffdh_{j+\nu}^k}{\ffh^k_{j+\nu}}}.
}
where $\ffdh^k_{j+\nu} := f''(2^{-k_0}t^k_{j+\nu})$.
It is obvious that
\ba\label{DFT}
\ffdh^k_{j+\nu} = f''(\tb)  + O(2^{-k_0-k}).
\ea
Note that the shape parameter $\gamma_{j_\nu}$ in the mask of
the NUCC scheme (see \eqref{GAM-EVEN})
is an approximation to $\gammah_{j_\nu}$.
The only difference is that the shape parameter  \eqref{GAM-H}
is given in terms of  the original function values.
Further, in view of Section \ref{SEC-CONST}, it is  obvious that
the operator $Q_\nu$ locally reproduces  the exponential polynomials
in the space $\EE_2$.

\def\ep{\epsilon}
\begin{lemma} \label{LEM-EST-Q}
For each $\nu=0,1$, let  $Q_\nu f$ be defined as  in \eqref{QF}
with $f \in W^3_\infty(\RR)$.
Let  $\tb :=2^{-k_0} t^{k+1}_{j_\nu}$ and
assume that
$|f(t)|\geq \tau $ for $|t-\tb|\leq 2^{-k_0-k}$ with a constant  $\tau>0$.
Then, we have the estimate
\eqn{ \label{QF-F}
| Q_\nu f(2^{-k_0}t^{k+1}_{j_\nu}) - f(2^{-k_0}t^{k+1}_{j_\nu}) |
\leq c_{f,\tau} 2^{-3(k_0+k)}
}
for some constant $c_{f,\tau} > 0$ depending on $f$ and $\tau$.
\end{lemma}
\begin{proof}
Noting that  the space $\EE_2$ in \eqref{E-SP}
can be written as
$\EE_2={\rm span}
       \{\cosh(\gammah_{j_\nu}(\cdot -\tb) ),\ \sinh(\gammah_{j_\nu}(\cdot -\tb) )\}$,
we define a function $G_f$ by
\eqn{ \label{G}
G_f = f(\tb) \cosh (\gammah_{j_\nu} (\cdot - \tb))
+ f'(\tb)\frac{\sinh (\gammah_{j_\nu} (\cdot - \tb))}{\gammah_{j_\nu}}.
}
Clearly, $G_f$ is in the space $\EE_2$ and
approximates $f$ in the sense that
$G_f(\tb) = f(\tb)$ and  $G_f'(\tb) = f'(\tb)$.
Moreover,
$$G_f''(\tb) = \gammah^2_{j_\nu} f(\tb) = \frac{\ffdh^k_{j+\nu}}{\ffh^k_{j+\nu}} f(\tb).$$
Then, taking the Taylor expansion of $f-G_f$ around $\tb$,
it follows that
\eqn{ \label{F-G-0}
f(t) - G_f (t) =
\Big(f''(\tb) - \frac{\ffdh^k_{j+\nu}}{\ffh^k_{j+\nu}} f(\tb) \Big)
 \frac{(t-\tb)^2}{2} +
(f-G_f)^{(3)}(\xi)\frac{(t-\tb)^3}{6}.
}
By \eqref{DFT},
$\ffdh^k_{j+\nu} = f''(\tb) + O(2^{-k_0-k})$.
Also,
$\ffh^k_{j+\nu} = f(\tb) + O(2^{-k_0-k})$.
Based on the definition of $\gammah_{j_\nu}$,
some elementary calculation reveals that
$|G_f^{(3)}(\xi)|\leq  c_f$ for some $c_f>0$.
Since $|f(t)|\geq \tau $ for $|t-\tb|\leq 2^{-k_0-k}$,
the identity \eqref{F-G-0} leads to the following bound:
\eqn{ \label{FG-111}
|f(t) - G_f(t)|
\leq c_{f,\tau} 2^{-3(k+k_0)}
}
with a constant $c_{f,\tau}>0$ depending on $f$ and $\tau$.
Now, since $G_f$ is in the space $\EE_2$,
due to the exponential reproducing property of $\LL_{0,\nu}$
and $\LL_{0,\nu}$, $G_f$ is represented as
$$G_f(\tb) = \LL_{0,\nu}(\tb) G_f(2^{-k_0}t^k_j)
            + \LL_{1,\nu}(\tb) G_f(2^{-k_0}t^k_{j+1}).$$
It leads to the expression
\eqn{ \label{QF-F-1}
\begin{aligned}
Q_\nu f(\tb) - f(\tb) &= Q_\nu f(\tb) - G_f(\tb)  \\
&= \LL_{0,\nu}(\tb) (f(2^{-k_0}t^k_j) - G_f(2^{-k_0}t^k_j)) + \LL_{1,\nu}(\tb) (f(2^{-k_0}t^k_{j+1}) - G_f(2^{-k_0}t^k_{j+1})).
\end{aligned}
}
Applying the estimate \eqref{FG-111} to \eqref{QF-F-1}
induces the required result \eqref{QF-F}.
\end{proof}

Using the above results,
we will  show that the NUCC scheme $\{ S_k \}$ can provide
an improved order of accuracy better than the classical methods.

\begin{lemma} \label{LEM-EST-S}
Let $\{ S_k \}$ be the NUCC scheme
acting on the initial data
$\ff^0 =  \{f(2^{-k_0}(n-\frac 12):n\in\ZZ\}$
with a function  $f\in W^3_\infty(\RR)$.
Let $\ffh^{k}=\{f(2^{-k_0}t^{k}_j):j\in\ZZ\}$  as given in \eqref{FH} and
$\tb :=2^{-k_0} t^{k+1}_{j_\nu}$ for $\nu=0,1$.
If $|f(t) |\geq \tau $ for $|t -\tb|\leq 2^{-k_0-k}$
with a  constant $\tau>0$,
then we get
\eqn{ \label{EST-1}
| S_k \ffh^{k}_{j_\nu} - \ffh^{k+1}_{j_\nu} | \leq c_{f,\tau} 2^{-(3k_0+2k)},
}
where $c_{f,\tau} > 0$ is a constant dependent upon $f$  and $\tau$.
\end{lemma}
\begin{proof}
Using the auxiliary function $Q_\nu f$ in \eqref{QF},
let us write
\eqn{ \label{REL-1}
| S_k \ffh^{k}_{j_\nu} - \ffh^{k+1}_{j_\nu} |
&\leq | S_k \ffh^{k}_{j_\nu} - Q_\nu f(\tb) |
       + | Q_\nu f(\tb) - \ffh^{k+1}_{j_\nu} |.
}
Then we first consider the case $j_\nu = 2j$, i.e., $\nu=0$.
To investigate the first term in the right-hand side of this inequality,
we see that
\eqn{ \label{REL-2}
S_k \ffh^{k}_{2j} - Q_0f(\tb) =
\big ( \bfa^{i,k}_0 - \LL_{0, 0}(\tb) \big) \ffh^k_j
+ \big( \bfa^{i,k}_{-2} - \LL_{1,0}(\tb) \big) \ffh^k_{j+1}
}
with the Lagrange functions $\LL_{n,0}$ in \eqref{QF}.
Recall that the mask $\bfa^{i,k}$  of the NUCC scheme and
$\LL_{0,0}(\tb)$ are
written in terms of the function
$\frac{\sinh (\frac 34\gamma_t)}{\sinh \gamma t}$
with a suitable $\gamma$ respectively.
By using the Maclaurin series expansion argument,
we have
\begin{align} \label{E-0}
\begin{split}
\big| \big ( \bfa^{j,k}_0 - \LL_{0,0}(\tb) \big) \ffh^k_j \big|
&\leq
c_1 \Big  | \Big( \frac{2^{-2k}\bd^k_{j}}{\ff^k_j}
- \frac{2^{-2(k_0+k)} \ffdh^k_j  }{ \ffh^k_{j} } \Big) \ffh^k_j \Big |
+ c_2 2^{-4(k_0+k)} \\
& =
c_1 2^{-2k} \Big  |\frac{\bd^k_{j}}{\ff^k_j} \ffh^k_j
- 2^{-2k_0} \ffdh^k_j \Big |  + c_2 2^{-4(k_0+k)},
\end{split}
\end{align}
for some constants $c_1, c_2 > 0$.
Here, $\bd^k = S^k_{\bfa}  \Delta \ff^0$ and
the data $\ff^0$ has the density $2^{-k_0}$. So,
obviously,
$$S^k_{\bfa} \Delta \ff^0_{j} = 2^{-2k_0} f''(2^{-k_0} t^{k+1}_{2j})
+ O(2^{-3k_0}).$$
Also, by \eqref{DFT},
$ \ffdh^k_j   =f''(2^{-k_0}t^{k+1}_{2j})+O(2^{-k_0})$.
Moreover, by Lemma \ref{LEM-EST-F},
$\ff^k_{j} = \ffh^{k}_{j} + O(2^{-2k_0})$.
Thus
a direct computation from \eqref{E-0} yields
\eq{
\big| \big ( \bfa^{i,k}_0 - \LL_{0,0}(\tb) \big) \ffh^k_j \big|
\leq c_{f,\tau} 2^{-3k_0 - 2k}
}
for a constant $c_{f,\tau} > 0$ depending on $f$ and $\tau$.
In a similar way, we can get
$\big| \big ( \bfa^{i,k}_{-2} - \LL_{1,0}(\tb) \big) \ffh^k_j \big|
\leq c_{f,\tau} 2^{-3k_0 - 2k}$
which estimates
the second term in the right-hand side of \eqref{REL-2}.
It concludes that
\eqn{ \label{BOUND-1}
| S_k \ffh^{k}_{j} - Q_0f(\tb)  | \leq c 2^{-3k_0 -2k}.
}
Due to  Lemma \ref{LEM-EST-Q},  the last term in \eqref{REL-1}
satisfies
$| Q_0f(\tb) - \ffh^{k+1}_{2j} | \leq c_{f,\tau} 2^{-3(k_0+k)}$.
This together with the bound \eqref{BOUND-1} implies the estimate \eqref{EST-1}.
The proof for the case $j_\nu = 2j+1$ can be done similarly.
So,  we finish the proof.
\end{proof}

We are now ready to  prove our main result.

\begin{theorem}\label{MAIN-TH}
Let $\{ S_k \}$ be the NUCC scheme
acting on the initial data of the form
$\ff^0 =  \{f(2^{-k_0}(n-\frac 12):n\in\ZZ\}$
with a function  $f\in W^3_\infty(\RR)$.
Put
$\tb = 2^{-k_0}t^{k+1}_j$ be a dyadic point.
Then if $|f(t)| \geq \tau$ with $\tau >0$  for  $|t -\tb| \leq 2^{-k_0-k}$,
then  the limit function
$\ff^\infty= \lim_{k\to\infty}S_k \cdots S_0 \ff^{0}$
satisfies the estimate
\eqn{ \label{APP-ORDER}
|\ff^\infty(\tb) - f(\tb)| \leq c_{f,\tau} 2^{-3k_0}
}
for a constant $c_{f,\tau} > 0$ dependent upon $f$ and
$\tau$ but independent of $k_0$.
\end{theorem}
\begin{proof}
Using the telescoping sum (as in the proof of Lemma \ref{LEM-EST-F})
and applying Lemma \ref{LEM-EST-S},
we can estimate the error between $\ff^{k+1}$ and $\ffh^{k+1}$ as
\eq{
\| \ff^{k+1} - \ffh^{k+1} \|_\infty \leq
c \sum_{\ell = 0}^{k} \| S_{\ell} \ffh^\ell - \ffh^{\ell+1} \|_\infty
\leq
c_{f,\tau} \sum_{\ell = 0}^{k} 2^{-(3k_0+2\ell)}
\leq c_{f,\tau}2^{-3k_0}
}
for a constant $c_{f,\tau} > 0$ depending on $f$ and
$\tau$.
Since this relation holds for any dyadic point $t^{k+1}_j$,
we can obtain the desired result.
\end{proof}

\begin{corollary}
Let the initial data be of the form $\ff^{0}=\{f(2^{-k_0}t^0_n):n\in\ZZ\}$
with $f \in W^3_\infty(\RR)$.
Assume that
$\|\ff^{0}\|_\infty\geq \tau $ for some $\tau>0$.
Under the same conditions of Theorem \ref{MAIN-TH}, we have
\eq{
\| \ff^\infty - f \|_\infty \leq c_{f,\tau} 2^{-3k_0}.
}
\end{corollary}

The following theorem treats the local approximation order of the
NUCC scheme for the (alternative) case that is suggested in \eqref{MASK-20}.
It can be proved in a similar fashion to the case of Theorem \ref{MAIN-TH}.

\begin{theorem}
Let $f \in W^3_\infty(\RR)$ and put $\tb = 2^{-k_0}t^{k+1}_j$
be a dyadic point.  Assume that $|f'(t)| \geq \tau$ with $\tau >0$
for  $|t -\tb| \leq 2^{-k_0-k}$.  Then, under the same conditions
and assumptions of Theorem \ref{MAIN-TH}, we have
\eqn{
|\ff^\infty(\tb) - f(\tb)| \leq c_{f,\tau} 2^{-3k_0}
}
for a constant $c_{f,\tau} > 0$.
\end{theorem}
\begin{proof}
The proof can be done by applying the same technique above.
\end{proof}

\begin{figure*}[t!]
\centering
\includegraphics[width=0.32\textwidth]{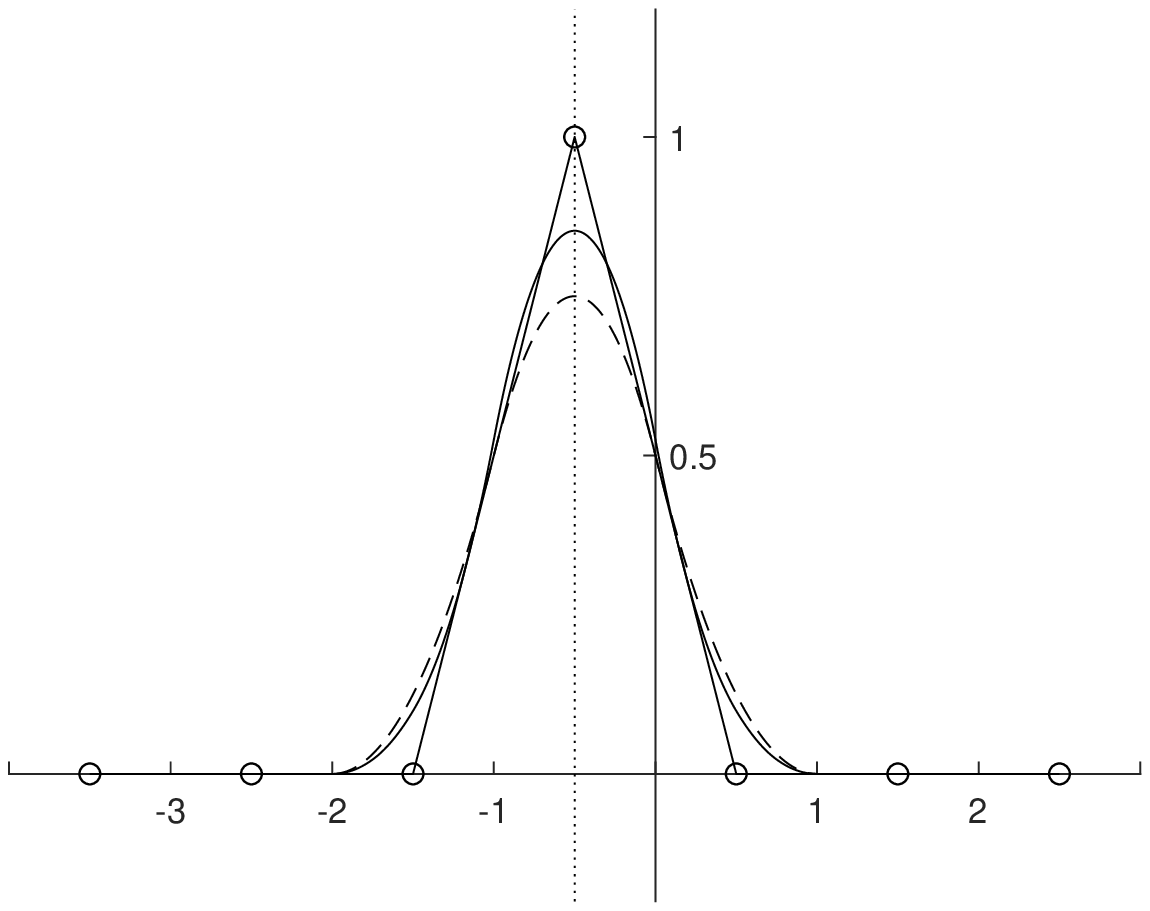}
\includegraphics[width=0.32\textwidth]{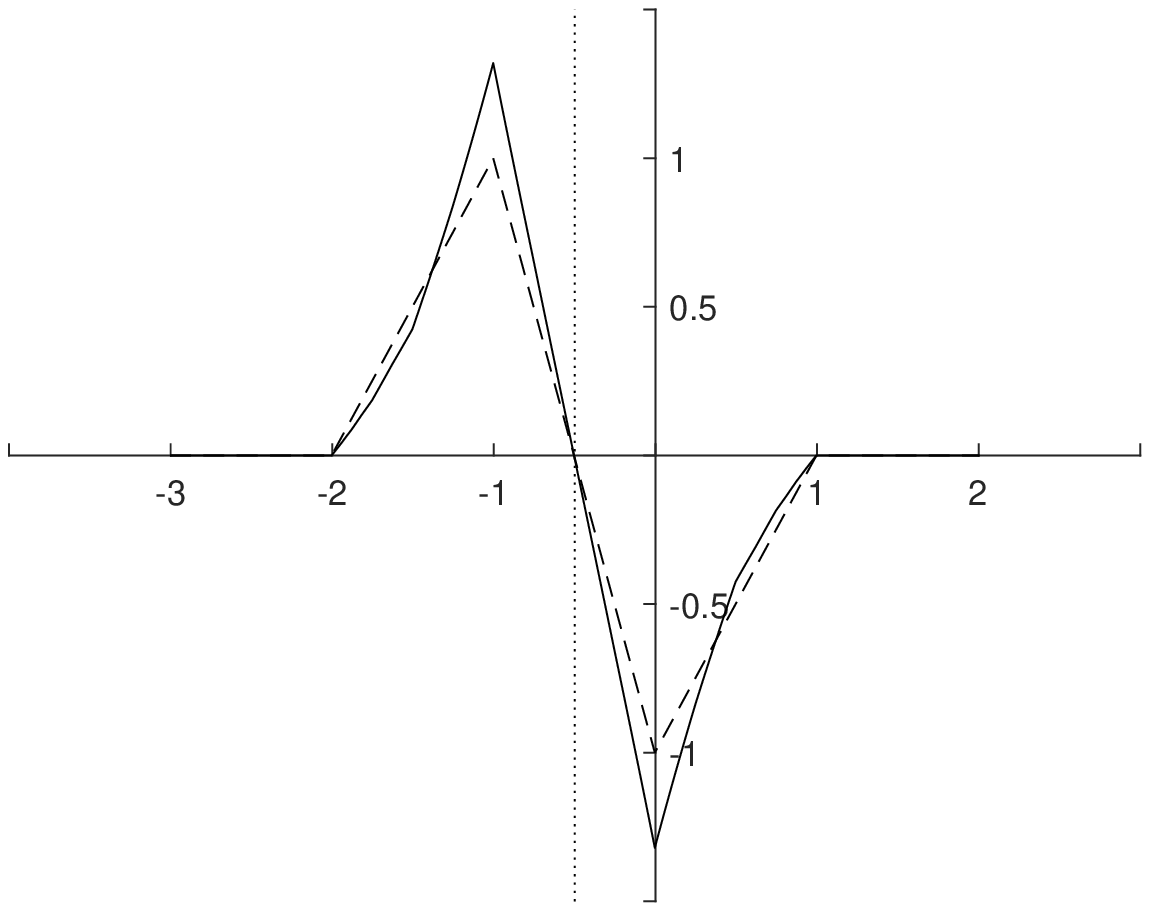}
\includegraphics[width=0.32\textwidth]{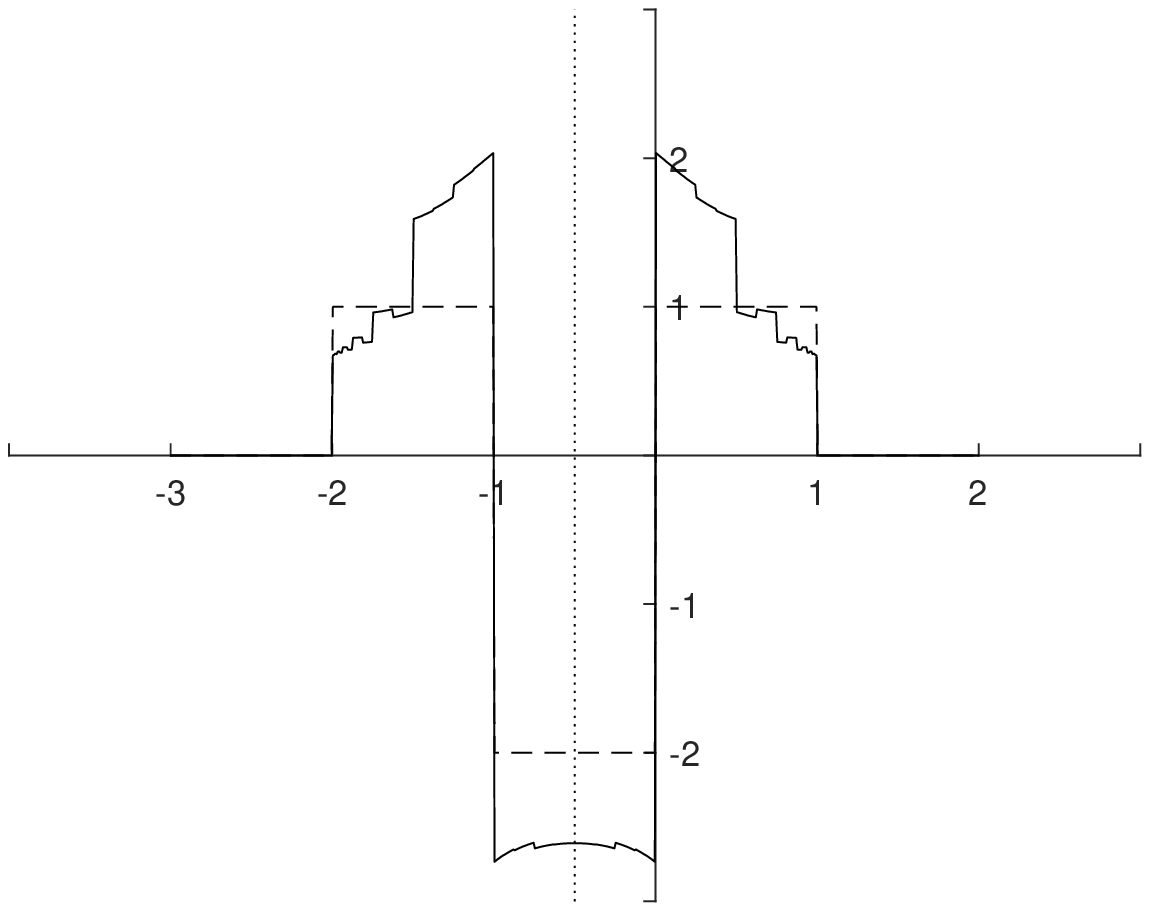}
\caption{Limit functions of the classical corner-cutting
algorithm (dashed line) and the proposed scheme (real line),
and their first- and second-order derivatives.
}
\label{FIG-S}
\end{figure*}

\section{Numerical examples} \label{SEC-NE}

In this section, we present some numerical examples to demonstrate
the performance of the NUCC scheme.
The first example verifies
that the proposed NUCC scheme is  $C^1$.
It supports the theoretical result in Section \ref{SEC-CS}.

\begin{example}
{\rm (Smoothness)
In this example, we put the initial data $\ff^0$ as
the Kronecker delta sequence, i.e.,
$\ff^0 = \{ \ff^0_j := \delta_{j,0} : j \in \ZZ \}$
where $\delta_{j,0}$ equals zero, if $j\ne 0$, and one
otherwise.
In order to get the limit function,
we recursively apply the refinement rule $\{S_k\}$  to $\ff^0$
using the mask given in \eqref{MASK}.
In this experiment, we set $\epsilon$
in the shape parameter $\gamma_j$
as $|\epsilon| = 1$.
We compare the limit function with the one generated
by the classical corner-cutting algorithm.
Figure \ref{FIG-S} depicts the two limit functions,
and their first- and second-order derivatives.

}
\end{example}

\begin{table}[t]
\centering
\caption{Comparison of the approximation order to the function defined in \eqref{T-FUNC}.}
\medskip
\begin{tabular}{ccccccc}
\hline
Density of initial data && \multicolumn{2}{c}{Max. Error} && \multicolumn{2}{c}{Approximation Order}\\
\cline{3-4} \cline{6-7}
 && Exp. B-spline & Proposed && Exp. B-spline & Proposed \\
\hline
1 && 8.6789E-02 & 5.0305E-02 &&     &     \\
$2^{-1}$ && 2.3629E-02 & 6.2276E-03 && 1.9 & 3.0 \\
$2^{-2}$ && 6.1175E-03 & 6.2632E-04 && 1.9 & 3.3 \\
$2^{-3}$ && 1.5701E-03 & 7.5863E-05 && 2.0 & 3.0 \\
$2^{-4}$ && 3.9306E-04 & 9.2633E-06 && 2.0 & 3.0 \\
$2^{-5}$ && 9.8297E-05 & 1.1537E-06 && 2.0 & 3.0 \\
$2^{-6}$ && 2.4576E-05 & 1.4397E-07 && 2.0 & 3.0 \\
$2^{-7}$ && 6.1442E-06 & 1.7986E-08 && 2.0 & 3.0 \\
$2^{-8}$ && 1.5360E-06 & 2.2479E-09 && 2.0 & 3.0 \\
$2^{-9}$ && 3.8394E-07 & 2.8126E-10 && 2.0 & 3.0 \\
\hline
\end{tabular}
\label{TAB-AO}
\end{table}

\begin{figure*}[t]
\centering
\includegraphics[width=0.36\textwidth]{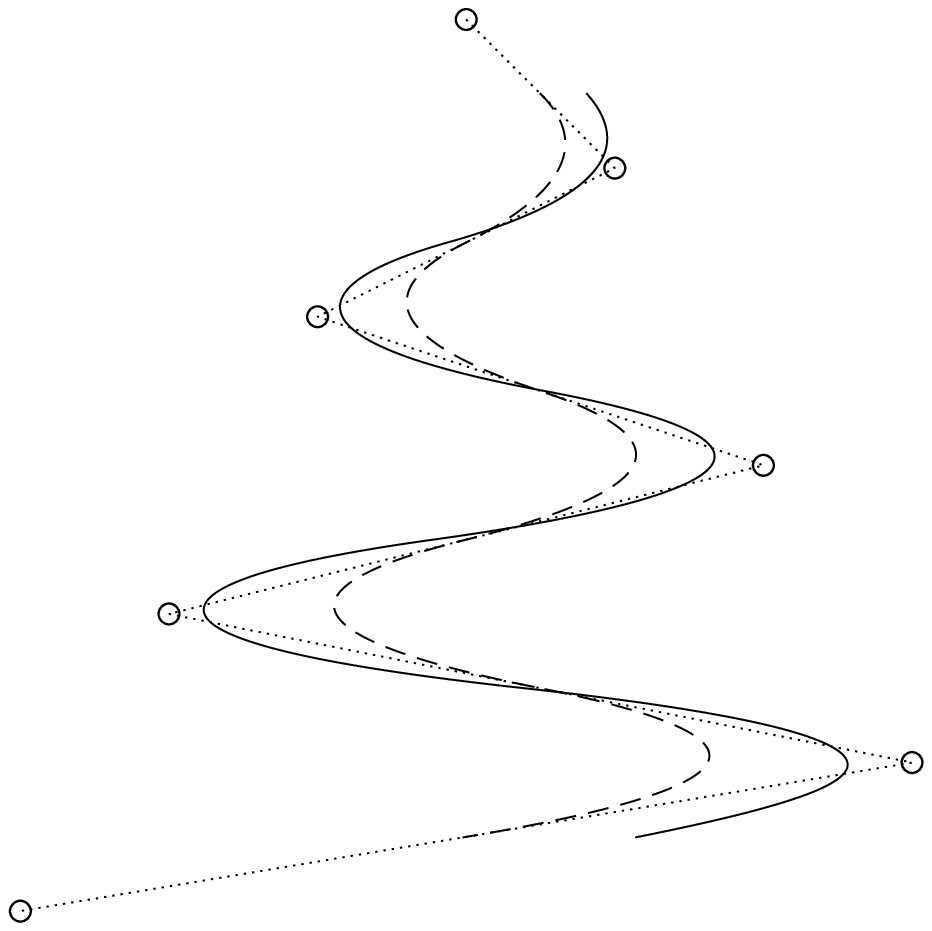}
\includegraphics[width=0.26\textwidth]{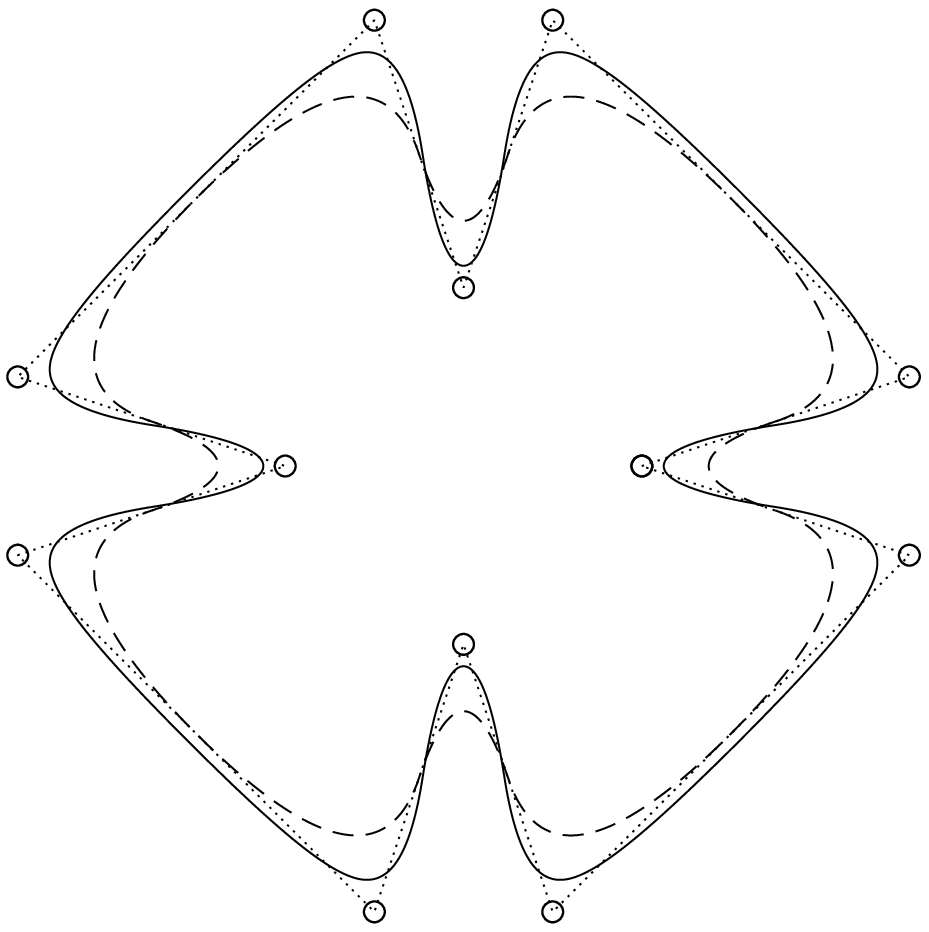} \qquad
\includegraphics[width=0.26\textwidth]{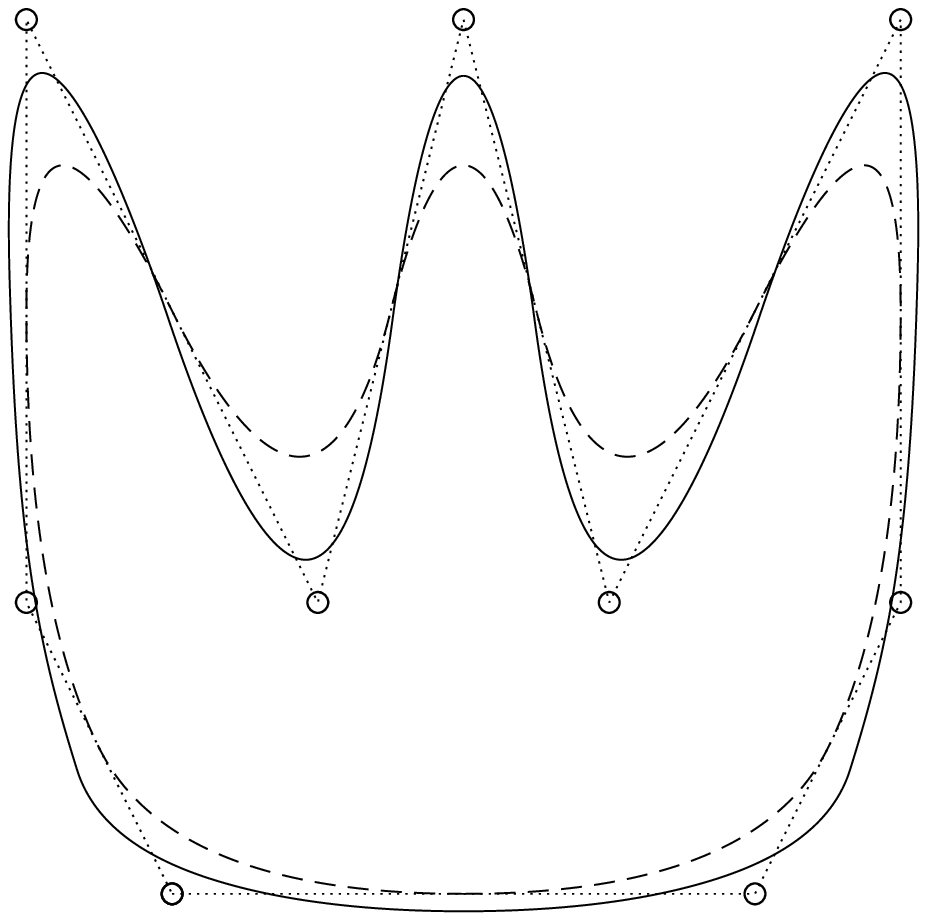}
\caption{Curve-fitting results:
classical corner-cutting method (dashed line)
and the new algorithm with $|\epsilon| = 1$ (real line).}
\label{FIG-CD}
\end{figure*}


The following example treats the approximation order of the NUCC scheme
discussed in Section \ref{SEC-AO},

\begin{example}
{\rm (Approximation order)
This  example tests the numerical  approximation order of the
NUCC scheme for
the (scaled one-dimensional Franke) function
\eqn{ \label{T-FUNC}
f(t):= \frac{3}{4} e^{-(\frac 98 t-2)^2/4}
+ \frac{3}{4} e^{-(\frac 98 t+1)^2/49}
+ \frac{1}{2} e^{-(\frac 98 t-7)^2/4} - \frac{1}{5} e^{-(\frac 98 t-4)^2}.
}
The initial data is sampled from the function $f$ with density $2^{-k_0}$
for $k_0=0,\ldots, 9$.
As verified in \cite{JKLY-13}, an exponential B-spline subdivision
can provide the approximation order at most {\em two}
when a suitable normalization factor is taken.
Thus, we employ the normalized exponential B-spline scheme of degree $2$
reproducing two exponential polynomials
$\{\exp(\gamma \cdot), \exp(-\gamma \cdot) \}$
with $\gamma = 1/2$.
The corresponding subdivision masks are given by
\eq{
\bfa^{[k]}_{-2} = \frac{\sinh(2^{-k-3})}{\sinh(2^{-k-1})},\quad
\bfa^{[k]}_{-1} = \frac{\sinh(3 \cdot 2^{-k-3})}{\sinh(2^{-k-1})},\quad
\bfa^{[k]}_{0} = \bfa^{[k]}_{-1},\quad
\bfa^{[k]}_{1} = \bfa^{[k]}_{-2}.
}
Table \ref{TAB-AO} displays the experimental results.
As one can see in Table \ref{TAB-AO},
the NUCC scheme achieves the third-order accuracy
while the exponential B-spline scheme attains the second order accuracy.

}
\end{example}

\begin{example}
{\rm (Curve fitting)
In this example, we compare the curve-fitting performance of
the classical corner-cutting method and the NUCC algorithm.
Figure \ref{FIG-CD} shows the resulting curves generated from several
initial data.
For this experiment, we used $|\epsilon| = 1$.
The limit curves of the proposed algorithm better follow the sharp corners
than those of the classical scheme.

}
\end{example}

\noindent{\bf Acknowledgements }
J. Yoon was supported in part by the National Research Foundation of Korea
under grant NRF-2020R1A2C1A01005894.




\end{document}